\documentclass[11pt, a4paper]{article}

\usepackage{amsmath, epsfig}
\usepackage{latexsym}
\usepackage{amssymb}
\usepackage{latexsym}
\usepackage{graphicx}

\def\tmk{T_{m,k}}
\def\tm2{T_{m,2}}
\def\lamh11{\lambda_{h,1,1}}

\def\th11{\t_{h,1,1}}

\newtheorem{theorem}{Theorem}[section]

\newtheorem{corollary}[theorem]{Corollary}
\newtheorem{proposition}[theorem]{Proposition}

\newtheorem{remark}{Remark}

\newtheorem{algorithm}{Algorithm}

\input amssym.def
\input amssym.tex

\long\def\delete#1{}

\def\qed{\hfill$\Box$\vspace{12pt}}

\def\b0{{\bf 0}}

\def\De{\Delta}

\def\b{\beta}
\def\d{\delta}

\def\l{\lambda}
\def\s{\sigma}
\def\t{\tau}
\def\ve{\varepsilon}

\def\sp{{\rm span}}
\def\mod{{\rm mod}}
\def\diam{{\rm diam}}

\newcommand{\be}{\begin{equation}}
\newcommand{\ee}{\end{equation}}
\newcommand{\bea}{\begin{eqnarray}}
\newcommand{\eea}{\end{eqnarray}}
\newcommand{\bean}{\begin{eqnarray*}}
\newcommand{\eean}{\end{eqnarray*}}

\title{Linear and cyclic distance-three labellings of trees}
\author{Deborah King, Yang Li, Sanming Zhou\footnote{smzhou@ms.unimelb.edu.au}\\ \\
{\small \it Department of Mathematics and Statistics}\\
{\small \it The University of Melbourne}\\
{\small \it Parkville, VIC 3010, Australia}\\ 
}

\date{May 30, 2014}

\begin{document}

\openup 0.5\jot\maketitle

\vspace{-0.5cm}

\smallskip
\begin{abstract}
Given a finite or infinite graph $G$ and positive integers $\ell, h_1, h_2, h_3$, an $L(h_1, h_2, h_3)$-labelling of $G$ with span $\ell$ is a mapping $f: V(G) \rightarrow \{0, 1, 2, \ldots, \ell\}$ such that, for $i = 1, 2, 3$ and any $u, v \in V(G)$ at distance $i$ in $G$, $|f(u) - f(v)| \geq h_i$. A $C(h_1, h_2, h_3)$-labelling of $G$ with span $\ell$ is defined similarly by requiring $|f(u) - f(v)|_{\ell} \ge h_i$ instead, where $|x|_{\ell} = \min\{|x|, \ell-|x|\}$. The minimum span of an $L(h_1, h_2, h_3)$-labelling, or a $C(h_1, h_2, h_3)$-labelling, of $G$ is denoted by $\l_{h_1,h_2,h_3}(G)$, or $\s_{h_1,h_2,h_3}(G)$, respectively. Two related invariants, $\l^*_{h_1,h_2,h_3}(G)$ and $\s^*_{h_1,h_2,h_3}(G)$, are defined similarly by requiring further that for every vertex $u$ there exists an interval $I_u$ $\mod~(\ell + 1)$ or $\mod~\ell$, respectively, such that the neighbours of $u$ are assigned labels from $I_u$ and $I_v \cap I_w = \emptyset$ for every edge $vw$ of $G$. 
A recent result asserts that the $L(2,1,1)$-labelling problem is NP-complete even for the class of trees. In this paper we study the $L(h, p, p)$ and $C(h, p, p)$ labelling problems for finite or infinite trees $T$ with finite maximum degree, where $h \ge  p \ge 1$ are integers. We give sharp bounds on $\lambda_{h,p,p}(T)$, $\l^*_{h,p,p}(T)$, $\s_{h, 1, 1}(T)$ and $\s^*_{h, 1, 1}(T)$, together with linear time approximation algorithms for the $L(h,p,p)$-labelling and the $C(h, 1, 1)$-labelling problems for finite trees. We obtain the precise values of these four invariants for complete $m$-ary trees with height at least 4, the infinite complete $m$-ary tree, and the infinite $(m+1)$-regular tree and its finite subtrees induced by vertices up to a given level. We give sharp bounds on $\s_{h,p,p}(T)$ and $\s^*_{h,p,p}(T)$ for trees with maximum degree $\Delta \le h/p$, and as a special case we obtain that $\s_{h,1,1}(T) = \s^*_{h,1,1}(T) = 2h + \De - 1$ for any tree $T$ with $\Delta \le h$. 
 

{\bf AMS 2010 mathematics subject classification:}~05C78, 05C15 
\end{abstract}

\section{Introduction}
\label{sec:int}

\textsf{A. Motivation.}~
In a radio communication network, a channel or a set of channels is required \cite{Hale} to assign to each transmitter such that the bandwidth used is minimized while interference between transmitters of geographical proximity is maintained at an acceptable level. From a combinatorial point of view, this fundamental problem is essentially an optimal labelling problem for the corresponding interference graph, which is defined to have transmitters as its vertices such that two vertices are adjacent if and only if the corresponding transmitters are geographically close to each other. In the case when each transmitter requires only one channel, we seek an assignment of a nonnegative integer (label) to each vertex such that for $i$ from $1$ to some given integer $d $, whenever two vertices are distance $i$ apart in the graph, the difference (in absolute value) between their labels is no less than a given separation. The existence of such an assignment is not a problem if sufficiently many channels are provided. However, since bandwidth is limited and costly, a major concern is to find the minimum span required among such channel assignments.  

\textsf{B. Linear and cyclic labellings.}~
The problem above can be modelled as follows. Given a finite or infinite undirected graph $G = (V(G), E(G))$ and a sequence of nonnegative integers $h_1, \ldots, h_d$, an {\em $L(h_1, \ldots, h_d)$-labelling} of $G$ with {\em span $\ell$} is a mapping $f: V(G) \rightarrow \{0, 1, 2, \ldots, \ell\}$ such that, for $i = 1, 2, \ldots, d$ and any $u, v \in V(G)$ with $d(u, v) = i$,
\be
\label{eq:lc}
|f(u) - f(v)| \geq h_i,
\ee
where $\ell$ is a positive integer and $d(u, v)$ denotes the distance in $G$ between $u$ and $v$. (In this paper an \textit{infinite graph} means a graph with countably infinitely many vertices. A graph is meant a finite graph unless stated otherwise.) In practical terms, the \textit{label} of $u$ under $f$, $f(u)$, is the channel assigned to the
transmitter corresponding to $u$. Without loss of generality we may always assume $\min_{v \in V(G)}f(v)=0$. The \textit{$\l_{h_1, \ldots, h_d}$-number} of $G$, denoted $\l_{h_1, \ldots, h_d}(G)$, is defined \cite{GY, Hale} to be the minimum span of an $L(h_1, \ldots, h_d)$-labelling of $G$. 
Equivalently, $\l_{h_1, \ldots, h_d}(G) = \min_{f} \sp(f)$, with minimum over all $L(h_1, \ldots, h_d)$-labellings of $G$, where $\sp(f) = \max_{v \in V(G)}f(v)$. In practice this parameter measures \cite{Hale} the minimum bandwidth required by the radio
communication network under constraints (\ref{eq:lc}).
 
The $L(h_1, \ldots, h_d)$-labelling problem above is a linear model in the sense that the metric involved is the $\ell_1$-metric. Its cyclic version was studied in \cite{HLS} with a focus on small $d$. A \emph{$C(h_1, \ldots, h_d)$-labelling} of a finite or infinite graph $G$ with {\em span $\ell$} is a mapping 
$f: V(G) \rightarrow \{0, 1, 2, \cdots, \ell-1\}$ such that, for $i = 1, 2, \ldots, d$ and any $u, v \in V(G)$ with $d(u, v) = i$, 
\be 
\label{eq:cc}
|f(u) - f(v)|_{\ell} \ge h_i,
\ee
where $|x-y|_{\ell} := \min\{|x-y|, \ell-|x-y|\}$ is the {\em $\ell$-cyclic distance} between integers $x$ and $y$. A $C(h_1, \ldots, h_d)$-labelling of $G$ with span $\ell$ exists for sufficiently large $\ell$. Define the {\em $\s_{h_1, \ldots, h_d}$-number} of $G$, denoted $\s_{h_1, \ldots, h_d}(G)$, to be the minimum integer $\ell$ such that $G$ admits a $C(h_1, \ldots, h_d)$-labelling with span $\ell$. Note that $\s_{h_1, \ldots, h_d}(G)$ thus defined agrees with $\s(G; h_1, \ldots, h_d)$ in \cite{HLS} and $c_{h_1, \ldots, h_d}(G)$ in \cite{L}, but is larger by one than $\s(G; h_1, \ldots, h_d)$ in \cite{CLZ}. As observed in \cite{Gamst, HLS}, this cyclic version allows the assignment of a set of channels $f(u), f(u)+\ell, f(u)+2\ell, \ldots$ to each transmitter $u$. The possibility of providing such multiple coverage is important  \cite{HLS} in large communication systems that serve many customers simultaneously. 

We will refer to (\ref{eq:lc}) and (\ref{eq:cc}) as the {\em $L(h_1, \ldots, h_d)$ conditions} and the {\em $C(h_1, \ldots, h_d)$ conditions $\mod~\ell$}, respectively. 

\textsf{C. Literature review.}~
The linear and cyclic labelling problems above are 
interesting in both theory and practical applications, and as such they have been
studied extensively over many years, especially in the case when $d=2$. In the simplest looking case when $d=1$, the $L(h_1)$-labelling problem becomes the classic vertex-colouring 
problem (which is important and difficult already) because $\l_{h_1}(G) = h_1(\chi(G)-1)$, where $\chi(G)$ is the
chromatic number of $G$. Another interesting and important special case is that $\l_{1, \ldots, 1}(G) = \chi(G^d)-1$, where $G^d$ is the $d$-th power of $G$ defined to have vertex set $V(G)$ such that two vertices are adjacent if and only if their distance in $G$ is at most $d$. In the case when $d=2$, many interesting
results on $\l_{h_1, h_2}$ have been obtained by various researchers for many families of finite 
graphs, especially when $(h_1, h_2) = (2,1)$; see e.g.~\cite{CK, CLZ1, CLZ, GM, GY, ZHOU, ZHOU1} and \cite{Cala} for an extensive bibliography. Griggs and Yeh \cite{GY} conjectured that $\l_{2,1}(G) \le \Delta^2$ for any graph $G$ with maximum degree $\Delta \ge 2$. This has been confirmed for several classes of graphs, including chordal graphs \cite{Sakai}, generalized Petersen graphs \cite{GM1}, Hamiltonian graphs with $\Delta \le 3$ \cite{Kang}, etc. Improving earlier results \cite{CK, GY}, Goncalves \cite{G} proved that $\l_{2,1}(G) \le \Delta^2+\Delta-2$ for any graph $G$ with $\Delta \ge 2$. A recent breakthrough in this direction by Havet, Reed and Sereni \cite{HRS} asserts that for any $h \ge 1$ there exists a constant $\Delta(h)$ such that every graph with maximum degree $\Delta \ge \Delta(h)$ satisfies $\l_{h, 1} \le \Delta^2$. In particular, the Griggs-Yeh conjecture is true for any graph with sufficiently large maximum degree.
 
In \cite{GY} it was proved that $\l_{2,1}(T) = \De+1$ or $\De+2$ for any tree $T$. A polynomial time algorithm for determining $\l_{2,1}(T)$ was given in \cite{CK}, while in general it was 
conjectured \cite{FKP} that the problem of determining 
$\l_{h_1, h_2}$ (where $h_1 > h_2 \ge 1$) for trees is NP-hard. In
\cite{CKKLY} it was proved that $\De + h - 1 \le \l_{h,1}(T)
\le \min\{\De + 2h-2, 2\De+h-2\}$ with both 
bounds attainable. In \cite{GM2}, for $h_1 \ge h_2$ the $\l_{h_1, h_2}$-number
was derived for infinite regular trees. In \cite{CPP}, for $h_1 < h_2$ the smallest integer 
$\l$ such that every tree of maximum degree $\De \ge 2$ admits an
$L(h_1,h_2)$-labelling of span at most $\l$ was studied.

In \cite{HLS} the $C(h_1, \ldots, h_d)$ labelling problem was studied for infinite triangular lattice, infinite square lattice and infinite line lattices with a focus on the cases $d=2, 3$. In \cite{L} it was proved that, for a graph $G$ with $n$ vertices, if its complement $G^c$ is Hamiltonian, then $\s_{2,1}(G) \le n$; otherwise $\s_{2,1}(G) = n + p_v(G^c)$, where $p_v(G^c)$ is the smallest number of vertex-disjoint paths covering $V(G^c)$. In \cite{CLZ} it was proved that for a Hamming graph $K_{n_1} \Box \cdots \Box K_{n_t}$ (where $n_1 \ge  \cdots \ge n_t$ and $\Box$ denotes the Cartesian product), if $n_1$ is sufficiently large relative to $n_2, \ldots, n_t$, then $\l_{2,1}$, $\l_{1,1}$ and their `no-hole' counterparts are all equal to $n_1n_2 - 1$, and $\s_{2,1}$, $\s_{1,1}$ and their `no-hole' counterparts are all equal to $n_1 n_2$. 

Relatively few results were known for the linear and cyclic labelling problems when $d=3$. 
In \cite{KRZ} King, Ras and Zhou obtained sharp bounds on $\l_{h,1,1}$ for trees. They asked whether for fixed $h \ge 2$ the $L(h,1,1)$-labelling problem for trees can be solved in polynomial time. Answering this question, in \cite{FGKLP} it was proved among others that the $L(2,1,1)$-labelling problem is NP-complete for trees. This indicates that even for trees the $L(h_1, h_2, h_3)$-labelling problem is difficult in general. In \cite{ZHOU2} the third author of the present paper obtained tight upper bounds on $\l_{h_1, h_2, h_3}(Q_n)$ by using a group-theoretic approach, where $Q_n$ is the $n$-dimensional cube. In \cite{CFTV} a linear time approximation algorithm to $L(h, 1, 1)$-label an outerplanar graph was given, using span $3\De+8$ when $h=1$ and $\De \ge 6$, and $3\De + 2h + 6$ when $h \ge 2$ and $\De \ge 4h+7$. In \cite{BP} an $O(d^2 h_1 n)$-time approximation algorithm was given for the $L(h_1, \ldots, h_d)$-labelling problem for trees, with performance ratio depending on $h_i$ and $\l(T^i)$ for $i=1, \ldots, d$. In a recent paper \cite{KSH} the $\l_{h_1, h_2, 1}$-number was determined for the direct product of $K_2$ and two other complete graphs under various conditions on $h_1$ and $h_2$.

\textsf{D. Elegant labellings.}~
Noting that the labellings producing the upper bounds in \cite{FGKLP, KRZ} assign an interval to the neighbourhood of each vertex, the following notion was introduced in \cite{FGKLP}. Let $f$ be an $L(h_1, \ldots, h_d)$ or $C(h_1, \ldots, h_d)$-labelling of $G$ with span $\ell$. We call $f$ {\em elegant} if for every vertex $u$, there exists an interval $I_u$ modulo $\ell + 1$ or $\ell$, respectively, such that $f(N(u)) \subseteq I_u$ and for every edge $uv$ of $G$, $I_u \cap I_v = \emptyset$. The minimum integer $\ell$ such that $G$ admits an elegant $L(h_1, \ldots, h_d)$-labelling, and an elegant $C(h_1, \ldots, h_d)$-labelling, respectively, of span $\ell$ is denoted by $\l^*_{h_1, \ldots, h_d}(G)$ and $\s^*_{h_1, \ldots, h_d}(G)$, respectively; these invariants are defined to be $\infty$ if $G$ does not admit an elegant labelling. Among others it was proved in \cite{FGKLP} that for trees $\l^*_{h, 1, 1}$ and $\s^*_{h, 1, 1}$ can be computed in polynomial time for any $h \ge 1$. We notice that the upper bound in \cite[Theorem 1]{KRZ} for $\l_{h,1,1}(T)$ is actually an upper bound for $\l^*_{h,1,1}(T)$.

\textsf{E. Main results in this paper.}~
In this paper we study the $L(h, p, p)$ and $C(h, p, p)$ labelling problems for finite or infinite trees $T$ with finite maximum degree $\De \ge 3$, where $h \ge  p \ge 1$ are integers. Four families of trees play a key role in our study, namely the complete $m$-ary tree $\tmk$ with height $k$, the infinite complete $m$-ary tree $T_{m,\infty}$, the infinite $(m+1)$-regular tree $\hat{T}_{m,\infty}$, and the finite subtree $\hat{T}_{m, k}$ of $\hat{T}_{m,\infty}$ induced by its vertices up to level $k$ from a fixed vertex, where $m, k \ge 2$ are integers. Since any finite tree is a subtree of some $\tmk$ or $\hat{T}_{m, k}$ and any infinite tree is a subtree of some $T_{m,\infty}$ or $\hat{T}_{m,\infty}$, understanding $\l_{h, p, p}$, $\l^*_{h, p, p}$, $\s_{h, p, p}$ and $\s^*_{h, p, p}$ for these special trees enables us to obtain upper bounds on these invariants for arbitrary trees.  

The main results in this paper are as follows. In Section \ref{sec:linear} we give sharp bounds (Theorem \ref{thm:h11}) on $\lambda_{h,p,p}(T)$ and $\l^*_{h,p,p}(T)$ and a $(1+\ve)$-factor approximation algorithm for the $L(h,p,p)$-labelling problem for finite trees, where $\ve = (\De - 1)/(\De_2 - 1)$ with $\De_2 = \max_{uv \in E(T)} (d(u)+d(v))$. If $T$ contains $T_{\De - 1, 2}$ or $\hat{T}_{\De - 1, 2}$ as a subtree, then we obtain better lower bounds for or the exact value of $\lambda_{h,p,p}(T)$ and $\l^*_{h,p,p}(T)$ (Theorem \ref{core:h11}). In particular, we obtain (Corollary \ref{thm:hpq}) the precise values of these invariants for $\tmk$ ($k \ge 4$), $T_{m,\infty}$, $\hat{T}_{m, k}$ and $\hat{T}_{m,\infty}$.
 
In Section \ref{sec:cyclic} we study $C(h, p, p)$-labellings and elegant $C(h, p, p)$-labellings of trees in the case when $h/p \ge \De$. We first give sharp bounds on $\s_{h,p,p}(T)$ and $\s^*_{h,p,p}(T)$ (Theorem \ref{thm:c}) under this condition, and as special cases we obtain that for $T = \tmk$, $T_{m,\infty}$, $\hat{T}_{m, k}$ or $\hat{T}_{m,\infty}$, $\s_{h,p,p}(T)$ and $\s^*_{h,p,p}(T)$ are between $2h+m$ and $2h+(m+1)p-1$ when $h \ge (m+1)p$. Theorem \ref{thm:c} implies $\s_{h,1,1}(T) = \s^*_{h,1,1}(T) = 2h + \De - 1$ and a linear time exact algorithm for computing $\s_{h,1,1}(T)$ for general $T$ when $h \ge \De$ (Corollary \ref{core:c}). In particular, for any $h \ge m+1$ and $T = \tmk$, $T_{m,\infty}$, $\hat{T}_{m, k}$ or $\hat{T}_{m,\infty}$, we obtain $\s_{h,1,1}(T) = \s^*_{h,1,1}(T) = 2h+m$ (Corollary \ref{core:c}). 

In Section \ref{sec:cyc} we focus on $C(h, 1, 1)$-labellings and elegant $C(h, 1, 1)$-labellings. We first give exact formulas for $\s_{h,1,1}(T)$ and $\s^*_{h,1,1}(T)$ for $T = \tmk$ ($k \ge 4$), $T_{m,\infty}$, $\hat{T}_{m, k}$ or $\hat{T}_{m,\infty}$ (Theorem \ref{thm:ck3}). This enables us to obtain sharp upper bounds on these invariants for general $T$ (Theorem \ref{thm:s11forD}) and subsequently a linear time $(4\De + 1)/(2\De + 4)$-approximation algorithm for the $C(h, 1, 1)$-labelling problem for finite trees (Theorem \ref{thm:1.4}). 
We conclude the paper by giving remarks and questions in Section \ref{sec:concl}.

The following notions arise naturally from our study in this paper: An $L(h, 1, 1)$ or $C(h, 1, 1)$-labelling $f$ with span $\ell$ is called {\em super elegant} if for every vertex $u$, $f(N(u))$ is an interval or a circular interval $\mod~\ell$, respectively. See Remarks \ref{re:super}, \ref{re:c} and \ref{re:c1} and the discussion in Section \ref{sec:concl}.

  
\section{Preliminaries}
\label{sec:notation}

\begin{table}
\begin{center}
\caption{Notation}

\bigskip
  \begin{tabular}{l|l}
   \hline
Notation & Definition \\ \hline \hline
$d(v)$ & Degree of a vertex $v$ in the graph under consideration\\ \hline
$N(v)$ & Neighbourhood of $v$ in the graph under consideration \\ \hline
$C(v)$ & Set of children of a vertex $v$ in the rooted tree under consideration \\ \hline 
$\De(T)$ or $\De$ & Maximum degree of a tree $T$ \\ \hline
$\De_{2}(T)$ or $\De_2$ & $\max_{uv \in E(T)} (d(u)+d(v))$\\ \hline 
$\diam(T)$ or $\diam$ & Diameter of a tree $T$\\ \hline
$T_{m, k}$ & Complete $m$-ary tree with height $k$, namely the $m$-ary tree in which \\
& all leaves are at level $k$ \\ \hline
$T_{m, \infty}$ & Infinite complete $m$-ary tree, namely the rooted tree with\\
& countably infinitely many vertices such that every vertex \\
& (including the root) has exactly $m$ children\\ \hline
$\hat{T}_{m, \infty}$ & The $(m+1)$-regular infinite tree, namely the tree with countably \\
& infinitely many vertices such that every vertex has degree $m+1$. We \\
& treat $\hat{T}_{m, \infty}$ as a rooted tree with root at any chosen vertex. \\ \hline
$\hat{T}_{m, k}$ & The finite subtree of $\hat{T}_{m, \infty}$ induced by its vertices up to level $k$\\ \hline
$u_0, u_i, u_{ij}, u_{i_1 \ldots i_{t}}$ & For a rooted tree $T$, we use $u_0$ to denote the root of $T$, $u_i$ ($1 \le i \le d(u_0)$) \\
& the children of $u_0$, $u_{ij}$ ($1 \le j \le d(u_i)-1$) the children of $u_i$, and in \\
& general $u_{i_1 \ldots i_{t} j}$ ($1 \le j \le d(u_{i_1 \ldots i_{t}})-1$) the children of $u_{i_1 \ldots i_{t}}$.\\
\hline
\end{tabular}
 \label{tbl:notation}
  \end{center}
\end{table}

We will frequently use the following observation: If $H$ is a subgraph of a graph $G$, then $\l_{h, p, p}(H) \le \l_{h, p, p}(G)$ and $\s_{h, p, p}(H) \le \s_{h, p, p}(G)$. Note that if $f$ is a $C(h_1, \ldots, h_d)$-labelling of $G$ with span $\ell$, then so is $f+c$ for any fixed integer $c$, where $f+c$ is defined by $(f+c)(u) = (f(u)+c)~\mod~\ell$. Thus without loss of generality we may assign $0$ to any chosen vertex when computing $\s_{h_1,  \ldots, h_d}(G)$. 

It is not hard to see \cite{HLS} that, for any $h_1 \ge \cdots \ge h_d$,  
\be
\label{eq:lvsc1}
\l_{h_1,  \ldots, h_d}(G) + 1 \le \s_{h_1, \ldots, h_d}(G) \le \l_{h_1, \ldots, h_d}(G) + h_1.
\ee 
Since every $C(h_1, \ldots, h_d)$-labelling with span $\ell + 1$ is an $L(h_1, \ldots, h_d)$-labelling with span $\ell$, we have (\cite[Proposition 2]{FGKLP}), for any $h_1 \ge \cdots \ge h_d$, 
\be
\label{eq:elegant}
\l_{h_1, \ldots, h_d}(G) + 1 \le \s_{h_1, \ldots, h_d}(G) \le \s^*_{h_1, \ldots, h_d}(G),
\ee
\be
\label{eq:elegant1}
\l_{h_1, \ldots, h_d}(G) + 1 \le \l^*_{h_1, \ldots, h_d}(G) + 1 \le \s^*_{h_1, \ldots, h_d}(G).
\ee
 
In a finite or infinite rooted tree, a vertex of degree 1 other than the root is called a {\em leaf}. The {\em level} of a vertex in a rooted tree is its distance to the root. The {\em height} of a rooted tree is the maximum level of a vertex or $\infty$ when the tree is infinite. A rooted tree $T$ is called a (full) {\em $m$-ary tree}, where $m \ge 1$ is an integer, if every non-leaf vertex of $T$ has exactly $m$ children (thus the root has degree $m$ and all other non-leaf vertices have degree $m+1$). 

As in \cite{KRZ}, for a graph $G$, we define $\De_2(G) := \max_{uv \in E(G)} (d(u)+d(v))$ and call an edge of $G$ a {\em heavy edge} if it attains $\De_2(G)$.

Table \ref{tbl:notation} summarizes the notation used in the paper. To avoid triviality we assume $m, k \ge 2$ in $\tmk$, $T_{m, \infty}$, $\hat{T}_{m, k}$ and $\hat{T}_{m, \infty}$.


\section{$L(h,p,p)$-labellings of trees}
\label{sec:linear}

\begin{theorem}
\label{thm:h11}
Let $h \ge p \ge 1$ be integers. Let $T$ be a finite tree with $\diam \ge 3$ or an infinite tree with finite maximum degree. Then 
\begin{equation}
\label{eq:h11}
\max\{(\De_2 - 1)p, h+(\De-1)p\} \le \lambda_{h,p,p}(T) \le \l^*_{h,p,p}(T) \le h + 2(\De - 1)p
\end{equation}
with both bounds attainable. Moreover, there exists a linear time $(1+\ve)$-factor approximation algorithm for the $L(h,p,p)$-labelling problem for the class of finite trees, where $\ve = (\De - 1)/(\De_2 - 1) < 1$.
\end{theorem}

This algorithm can be described as follows. Given a finite or infinite tree $T$, we may treat it as a rooted tree with root at any chosen vertex $u_0$. Define
$$
A_0 := \{0, p, 2p, \ldots, (\De - 1)p\}
$$ 
$$
A_1 := \{h+(\De - 1)p, h+\De p, \ldots, h+(2\De - 2)p\}.
$$
We construct a labelling $f: V(T) \rightarrow \{0, 1, \ldots, h+(2\De - 2)p\}$ by the following algorithm.

\begin{algorithm}
\label{alg:lambda}
1. Set $f(u_0) = 0$.

2. Assign distinct labels from $A_1$ to the children of $u_0$.

3. Inductively, assume that for some $t \ge 1$ all vertices of $T$ up to level $t$ have been labelled. Then, for every vertex $u_{i_1 \ldots i_{t}}$ of $T$ at level $t$, 
\begin{itemize}
\item if $t$ is odd, then assign distinct labels from $A_0 \setminus \{f(u_{i_1 \ldots i_{t-1}})\}$ to the children of $u_{i_1 \ldots i_{t}}$; 
\item if $t$ is even, then assign distinct labels from $A_1 \setminus \{f(u_{i_1 \ldots i_{t-1}})\}$ to the children of $u_{i_1 \ldots i_{t}}$.
\end{itemize}
If $T$ is infinite, repeat step 3 indefinitely; otherwise, stop until all vertices are labelled.  
\end{algorithm}

\begin{proof}\textbf{of Theorem \ref{thm:h11}}~
We prove the lower bound first. Let $xy$ be a heavy edge of $T$ and $f$ an arbitrary $L(h,p,p)$-labelling of $T$. Since any two vertices in $N(x) \cup N(y)$ are within distance three apart in $T$, they should receive labels that differ by at least $p$. Since $|N(x) \cup N(y)| = \De_2$, it follows that $\sp(f) \ge (\De_2 - 1)p$. 
 
Let $z$ be a vertex of $T$ with $d(z) = \De$. Since the $\De$ vertices in $N(z)$ are pairwise distance two apart, we have $\max f(N(z)) \ge (\De - 1)p$. 
Moreover, since $z$ is adjacent to these $\De$ vertices, we must have 
$\max \left(f(N(z)) \cup \{f(z)\}\right) \ge h + (\De-1)p$
and so $\sp(f) \ge h + (\De-1)p$. Now that $\sp(f) \ge\max\{(\De_2 - 1)p, h + (\De-1)p\}$ for any $L(h,p,p)$-labelling $f$ of $T$, the lower bound in (\ref{eq:h11}) follows. This bound is achieved by, for example, the unique tree (up to isomorphism) with five vertices and diameter three, for any positive integers $h$ and $p$ with $2p \le h \le 3p$.   

To prove the upper bound we show that Algorithm \ref{alg:lambda} produces an $L(h,p,p)$-labelling of $T$. It is possible to carry out step 3 in this algorithm for all $t \ge 1$ because $|A_0| = |A_1| = \De$ and every vertex at level $t \ge 1$ has at most $\De - 1$ children. Note that, for distinct vertices $u_{i_1 \ldots i_{t}}$ and $u_{j_1 \ldots j_{t}}$ at level $t$, in step 3 we label the vertices in $C(u_{i_1 \ldots i_{t}})$ and that in $C(u_{j_1 \ldots j_{t}})$ independently. This does not violate the $L(h, p, p)$-conditions because any vertex in the former is at least distance four apart from any vertex in the latter. The vertices in $C(u_{i_1 \ldots i_{t}})$ are pairwise distance two apart. Apart from these vertices, the only other vertices which are within distance three from the vertices in $C(u_{i_1 \ldots i_{t}})$ are: $u_{i_1 \ldots i_{t}}$; $u_{i_1 \ldots i_{t-1}}$; and the vertices in $(C(u_{i_1 \ldots i_{t-1}}) \setminus \{u_{i_1 \ldots i_{t}}\}) \cup \{u_{i_1 \ldots i_{t-2}}\}$, which are distance three apart from any vertex in $C(u_{i_1 \ldots i_{t}})$. 

Observe that (i) the labels in $A_0$ ($A_1$, respectively) pairwise differ by at least $p$, (ii) the difference (in absolute value) between any label in $A_0$ and any label in $A_1$ is at least $h$, (iii) $u_{i_1 \ldots i_{t-1}}$ and the vertices in $C(u_{i_1 \ldots i_{t}})$ all use labels from $A_0$ if $t \ge 1$ is odd and from $A_1$ if $t \ge 2$ is even, and (iv) the vertices in $C(u_{i_1 \ldots i_{t}})$ and that in $C(u_{i_1 \ldots i_{t-1}}) \cup \{u_{i_1 \ldots i_{t-2}}\}$ use labels from different set of $A_0$ and $A_1$. Using these and by induction, it can be verified that Algorithm \ref{alg:lambda} produces an $L(h, p, p)$-labelling of $T$ with span $h + 2(\De - 1)p$. Moreover, from the way the labels are assigned it is clear that this is an elegant $L(h,p,p)$-labelling of $T$. Therefore, $\l^*_{h,p,p}(T) \le h + 2(\De - 1)p$, which together with $\lambda_{h,p,p}(T) \le \l^*_{h,p,p}(T)$ (by (\ref{eq:elegant1})) proves the upper bound in (\ref{eq:h11}). We will see in Corollary \ref{thm:hpq} that this upper bound is attained by $\tmk$ ($k \ge 4$), $T_{m, \infty}$, $\hat{T}_{m, k}$ ($k \ge 2$) and $\hat{T}_{m, \infty}$.

By (\ref{eq:h11}) the performance ratio of Algorithm \ref{alg:lambda} is at most
$$
\frac{h+2(\De-1)p}{\max\{(\De_2 - 1)p, h+(\Delta-1)p\}}.
$$
It can be verified that, no matter whether $h \le (\De_2 - \De)p$ or $h \ge (\De_2 - \De)p$, this ratio is at most $1+\ve$. Thus Algorithm \ref{alg:lambda} is a $(1+\ve)$-factor approximation algorithm for the $L(h,p,p)$-labelling problem for finite trees. Obviously it runs in linear time.  
\end{proof}
\qed 

\begin{remark}
{\em
An $O(d^2 h_1 n)$-time approximation algorithm was given in \cite{BP} for the $L(h_1, \ldots, h_d)$-labelling problem for trees. This can be specified to give an $O(9hn)$-time approximation algorithm for the $L(h, p, p)$-labelling problem for trees. The performance ratio \cite[Theorem 4]{BP} of this algorithm relies on $(h, p)$ and the chromatic numbers $\chi(T)\, (= 2), \chi(T^2)$ and $\chi(T^3)$, and it may be as big as $6$ in some cases. The ratio in Theorem \ref{thm:h11} is in general smaller and does not require information about $\chi(T^2)$ and $\chi(T^3)$.
}
\end{remark}

The lower bound in (\ref{eq:h11}) can be improved when $T$ contains $T_{\De-1, 2}$ or $\hat{T}_{\De-1, 2}$ as a subtree, as we now show in the following result. 

\begin{theorem}
\label{core:h11}
Let $h \ge p \ge 1$ be integers, and $T$ a finite tree with $\diam \ge 3$ or an infinite tree with finite maximum degree. 
\begin{itemize}
\item[\rm (a)] If $T$ contains $T_{\De-1, 2}$ as a subtree, then   
\be
\label{eq:subtree}
h + (2\Delta-3)p \le \l_{h,p,p}(T) \le \l^*_{h,p,p}(T) \le h + 2(\Delta-1)p;
\ee
\item[\rm (b)] if $T$ contains $\hat{T}_{\De-1, 2}$ as a subtree, then    
\be
\label{eq:subtree1}
\l_{h,p,p}(T) = \l^*_{h,p,p}(T) = h + 2(\Delta-1)p.
\ee
\end{itemize}
\end{theorem}
 
\begin{proof}
(a) It suffices to prove that, for any $m \ge 2$, 
\be
\label{eq:hppm2}
\l_{h,p,p}(T_{m,2}) \ge h + (2m-1)p.
\ee
In fact, if $T$ contains $T_{\De-1, 2}$ as a subtree, then by (\ref{eq:hppm2}) we have $\l_{h,p,p}(T) \ge \l_{h,p,p}(T_{\De-1, 2}) \ge h + (2\Delta-3)p$, which together with the upper bound in (\ref{eq:h11}) implies (\ref{eq:subtree}). 

We now prove (\ref{eq:hppm2}).
Consider an arbitrary $L(h,p,p)$-labelling $f$ of $\tm2$. Denote 
$$
f(u_i) = a_i,\; 0 \le i \le m; \quad f(u_{ij}) = a_{ij},\;1 \le i, j \le m.
$$ 
Since $u_1, \ldots, u_m$ are symmetric, and $u_{11}, \ldots, u_{1m}$ are symmetric, without loss of generality we may assume 
$$
a_1 < \cdots < a_m; \quad a_{11} < \cdots < a_{1m}.
$$ 
Then $\min f(N(u_1)) =  \min\{a_0, a_{11}\}$ and $\max f(N(u_1)) = \max\{a_0, a_{1m}\} \ne a_1$. 

\medskip
\textsf{Case 1:}~$a_1 >  \max\{a_0, a_{1m}\}$.~Since $u_1$ is adjacent to $u_0$ and $u_{1m}$, we have $|a_1 - a_0| \ge h$ and $|a_1 - a_{1m}| \ge h$, which together with our assumption $a_1 > \max\{a_0, a_{1m}\}$ imply $a_1 \ge  \max\{a_0, a_{1m}\} + h$. Since $|N(u_1)| = m+1$ and any two labels in $f(N(u_1))$ should differ by at least $p$, we have $\max\{a_0, a_{1m}\} \ge mp$ and so $a_1 \ge h+mp$. Since $\min f(N(u_0)) = a_1 >  \max\{a_0, a_{1m}\}$ and any two labels in $f(N(u_0))$ differ by at least $p$, from $a_1 \ge h+mp$ we obtain $a_m \ge h + (2m-1)p$. Therefore, $\sp(f) \ge h + (2m-1)p$.

\medskip
\textsf{Case 2:}~$a_1 <  \max\{a_0, a_{1m}\}$.~Since any two vertices in $N(u_0) \cup N(u_1)$ are within distance three apart in $\tm2$, the labels assigned to the $2m+1$ vertices of $N(u_0) \cup N(u_1)$ should pairwise differ by at least $p$. Moreover, since $u_1$ is adjacent to $u_0, u_{11}, \ldots, u_{1m}$, the difference between $a_1$ and each label among $a_0, a_{11}, \ldots, a_{1m}$ that is closest to $a_1$ should be at least $h$. Since there is at least one such label closest to $a_1$, we have $\sp(f) \ge h + 2mp > h + (2m-1)p$, no matter whether $a_{1m} < a_1 < a_0$, $a_0 < a_1 < a_{1m}$, or $a_1 < a_0$ and $a_1 < a_{1m}$.

In summary, we have proved $\sp(f) \ge h + (2m-1)p$ for any $L(h,p,p)$-labelling $f$ of $\tm2$. Hence (\ref{eq:hppm2}) follows. 

(b) Similar to (a), one can prove $\l_{h,p,p}(\hat{T}_{m, 2}) \ge h + 2mp$ for any $m \ge 2$.
This together with (\ref{eq:h11}) implies (\ref{eq:subtree1}). 
\end{proof}
\qed
 
\begin{corollary}
\label{thm:hpq}
Let $h \ge p \ge 1$ and $m, k \ge 2$ be integers. Then, for $T = \tmk$ ($k \ge 4$), $T_{m, \infty}$, $\hat{T}_{m, k}$ or $\hat{T}_{m, \infty}$, 
\be
\label{eq:hat}
\l_{h,p,p}(T) = \l^*_{h,p,p}(T) = h+2mp.
\ee
In addition, 
\begin{equation}
\label{eq:thmhpq}
\lambda_{h,p,p}(\tm2) = \lambda^*_{h,p,p}(\tm2) = h+(2m-1)p
\end{equation}
\begin{equation}
\label{eq:thmhpq1}
\l_{h,p,p}(T_{m, 3}) = \l^*_{h,p,p}(T_{m, 3}) = \max\{h + (2m-1)p, (2m+1)p\}.
\end{equation}
\end{corollary}

\begin{proof}
Note that in each case $T$ contains $\hat{T}_{m, 2}$ as a subtree. (We need $k \ge 4$ for $\tmk$ to contain $\hat{T}_{m, 2}$.) Thus by (b) of Theorem \ref{core:h11}, $h+2mp = \l_{h,p,p}(\hat{T}_{m, 2}) \le \l_{h,p,p}(T) \le \l^*_{h,p,p}(T)$. This together with the upper bound in (\ref{eq:h11}) yields (\ref{eq:hat}). 

When applying Algorithm \ref{alg:lambda} to $\tm2$, in step 2  we only need labels in $A_1 \setminus \{h+2mp\}$ for the children of $u_0$. Hence $\l^*_{h,p,p}(\tm2) \le h+(2m - 1)p$. This together with (\ref{eq:hppm2}) implies (\ref{eq:thmhpq}). 

Now we prove (\ref{eq:thmhpq1}). Since the $2m+2$ vertices in $N(u_1) \cup N(u_{11})$ have mutual distance at most $3$, we have $\l_{h,p,p}(T_{m, 3}) \ge (2m+1)p$. We also have $\l_{h,p,p}(T_{m, 3}) \ge \l_{h,p,p}(\tm2) = h + (2m-1)p$ by (\ref{eq:thmhpq}). If $h \ge 2p$, then we assign: $mp$ to $u_0$; $h+(m+i-1)p$ to $u_i$, $1 \le i \le m$; distinct labels from $\{0, p, \ldots, (m-1)p\}$ to the children of $u_i$, $1 \le i \le m$; and distinct labels from $\{h+(m-1)p, h+mp, \ldots, h+(2m-1)p\} \setminus \{h+(m+i-1)p\}$ to the children of $u_{ij}$, $1 \le i, j \le m$. Since $h \ge 2p$, this defines an elegant $L(h, p, p)$-labelling of $T_{m, 3}$ and so $\l^*_{h,p,p}(T_{m, 3}) \le h + (2m-1)p$. Hence $\l_{h,p,p}(T_{m, 3}) = \l^*_{h,p,p}(T_{m, 3}) = h + (2m-1)p$. If $h \le 2p$, then we assign: $mp$ to $u_0$; $(m+i+1)p$ to $u_i$, $1 \le i \le m$; distinct labels from $\{0, p, \ldots, (m-1)p\}$ to the children of $u_i$, $1 \le i \le m$; and distinct labels from $\{(m+1)p, (m+2)p, \ldots, (2m+1)p\} \setminus \{(m+i+1)p\}$ to the children of $u_{ij}$, $1 \le i, j \le m$. Since this defines an elegant $L(h, p, p)$-labelling of $T_{m, 3}$, we have $\l^*_{h,p,p}(T_{m, 3}) \le (2m+1)p$, leading to $\l_{h,p,p}(T_{m, 3}) = \l^*_{h,p,p}(T_{m, 3}) = (2m+1)p$.
\end{proof}
\qed
 
\begin{remark}
\label{re:super}
{\em
By (\ref{eq:hat}), $\tmk$ ($k \ge 4$), $\hat{T}_{m, k}$, $T_{m, \infty}$ and $\hat{T}_{m, \infty}$ achieve the upper bound in (\ref{eq:h11}) as claimed earlier. This implies that for these trees Algorithm \ref{alg:lambda} produces optimal elegant $L(h,p,p)$-labellings as well as optimal $L(h,p,p)$-labellings. In the special case when $p=1$, these labellings $f$ are super elegant in the sense that $f(N(v))$ is an integer interval for every vertex $v$. 
}
\end{remark}


\section{$C(h,p,p)$-labellings}
\label{sec:cyclic}

A {\em circular interval} $\mod~\ell$ is defined as a set of integers of the form 
$$
[a, b]_{\ell} := \{a~(\mod~\ell), a+1~(\mod~\ell), \ldots, b~(\mod~\ell)\},
$$ 
where $a, b$ are integers possibly with $a \ge b$. For instance, $[1,3]_7 = \{1, 2, 3\}$, $[5,2]_7 = \{5, 6, 0, 1, 2\}$ and $[4,0]_7 = \{4, 5, 6, 0\}$ are all circular intervals $\mod~7$. 
A {\em $p$-set} $\mod~\ell$ is defined as a set of integers of the form 
$$
c + p[a, b]_{\ell} := \{c+ap~(\mod~\ell), c+(a+1)p~(\mod~\ell), \ldots, c+bp~(\mod~\ell)\},
$$ 
where $a, b$ and $c$ are integers.  

The following is the main result in this section. Note that in this theorem the gap $p-1$ between the upper and lower bounds is independent of $T$, and the lower bound does not require the condition $h \ge \De p$ as we will see in its proof.

\begin{theorem}
\label{thm:c}
Let $T$ be a finite tree with $\diam \ge 3$ and $\De \ge 3$, or an infinite tree with finite maximum degree $\De \ge 3$. Then, for any integers $h, p \ge 1$ such that $h \ge \De p$, we have
\be
\label{eq:c}
2h + (\De - 1)p \le \s_{h,p,p}(T) \le \s^*_{h,p,p}(T) \le 2h + \De p - 1.
\ee
In particular, if $T = \tmk, T_{m, \infty}, \hat{T}_{m, k}$ or $\hat{T}_{m, \infty}$ and $h \ge (m+1)p$, then
\be 
\label{eq:d}
2h+mp \le \s_{h,p,p}(T) \le \s^*_{h,p,p}(T) \le 2h+(m+1)p-1.
\ee
\end{theorem}

\begin{figure}[ht]
\centering
\includegraphics*[height=9cm]{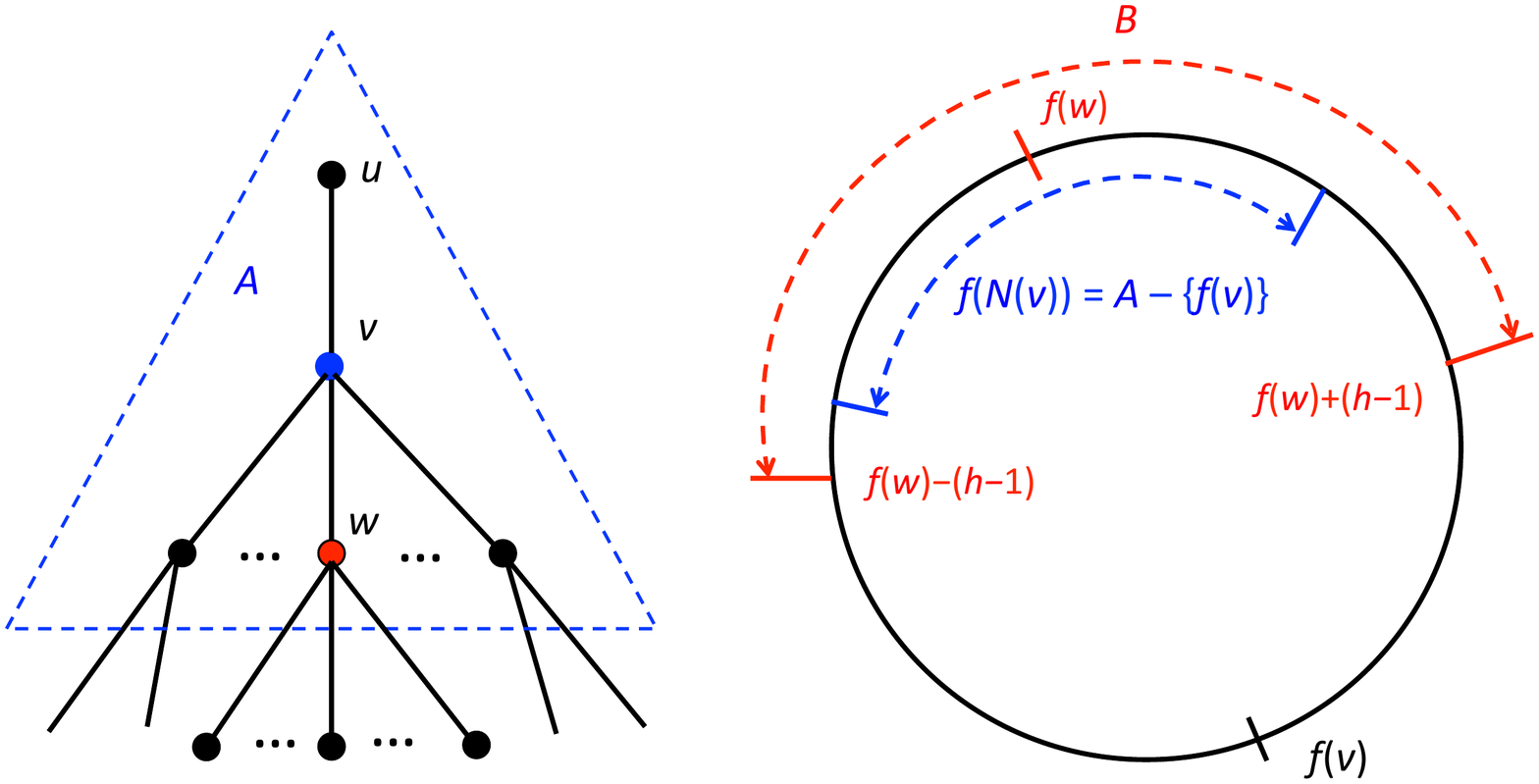}
\vspace{-2.0cm}
\caption{\small Proof of Theorem \ref{thm:c}, where $u=u_{i_1 \ldots i_{t-2}}, v=u_{i_1 \ldots i_{t-1}}$ and $w=u_{i_1 \ldots i_{t}}$.} \label{fig:ABC}
\end{figure} 

\begin{proof}
We prove the lower bound first. Let $u$ be a vertex of $T$ with $d(u) = \De$. Consider a $C(h,p,p)$-labelling $f$ of $T$ and suppose $f$ has span $\ell$. Without loss of generality we may assume $f(u) = 0$. Then $f(N(u)) \subseteq [h, \ell-h]_{\ell}$. Since the $\De$ vertices in $N(u)$ need labels with mutual separation at least $p$, we have $h + (\De - 1)p \le \ell - h$, that is, $\ell \ge 2h+(\De-1)p$. Since this holds for every $C(h,p,p)$-labelling of $T$, the lower bound in (\ref{eq:c}) follows. 

The inequality in the middle of (\ref{eq:c}) follows from (\ref{eq:elegant}). 
To prove the upper bound in (\ref{eq:c}), it suffices to prove that, if $m \ge 2$ and $h \ge (m+1)p$, then
\be 
\label{hat4}
\s^*_{h,p,p}(\hat{T}_{m, \infty}) \le 2h+(m+1)p-1.
\ee
In fact, since $T$ has maximum degree $\De$, it can be viewed as a subtree of $\hat{T}_{\De-1, \infty}$ (with root at any chosen vertex). Hence, if $h \ge \De p$, then $\s^*_{h,p,p}(T) \le \s^*_{h,p,p}(\hat{T}_{\De-1, \infty}) \le 2h + \De p - 1$ by (\ref{hat4}). 

Set $\ell: = 2h+(m+1)p-1$ in the rest of this proof. 
We now prove (\ref{hat4}) under the assumption $h \ge (m+1)p$ by constructing an elegant $C(h,p,p)$-labelling $f: V(\hat{T}_{m, \infty}) \rightarrow \{0, 1, \ldots, \ell - 1\}$ of $\hat{T}_{m, \infty}$ in the following way. First, define
\begin{itemize}
\item $f(u_0) = 0$;
\item $f(u_i) = h+(i-1)p$, $1 \le i \le m+1$;
\item $f(C(u_i)) = \left((2h+(p-1)) + p[i-1, m+i-1]_{\ell}\right) \setminus \{0\} = \{2h+ip-1, 2h+(i+1)p-1, \ldots, 2h+mp-1\} \cup \{p, 2p, \ldots, (i-1)p\}$, $1 \le i \le m+1$.
\end{itemize}
In the last step we can assign distinct labels from $\left((2h+(p-1)) + p[i-1, m+i-1]_{\ell}\right) \setminus \{0\}$ to the children of $u_i$ in any bijective manner. 
One can verify that the $C(h,p,p)$-conditions $\mod~\ell$ are satisfied among the vertices of $\hat{T}_{m, \infty}$ up to level 2. 

Note that $f(N(u_0))$ is a $p$-set $\mod~\ell$ and $f(N(u_i)) = f(C(u_i)) \cup \{0\}$ is a $p$-set $\mod~\ell$ for $1 \le i \le m+1$. Assume inductively that, for some $t \ge 2$, all vertices of $\hat{T}_{m, \infty}$ up to level $t$ have been labelled such that:
\begin{itemize}
\item[] $P_t$: the $C(h,p,p)$-conditions $\mod~\ell$ are satisfied among vertices up to level $t$;  
\item[] $Q_t$: for every vertex $u_{i_1 \ldots i_{t-1}}$ at level $t-1$, $f(N(u_{i_1 \ldots i_{t-1}}))$ is a $p$-set $\mod~\ell$. 
\end{itemize}
Based on these we now prove that we can label all vertices of $\hat{T}_{m, \infty}$ at level $t+1$ such that $P_{t+1}$ and $Q_{t+1}$ are satisfied. 

As seen in the proof of Theorem \ref{thm:h11}, for distinct $u_{i_1 \ldots i_{t}}$ and $u_{j_1 \ldots j_{t}}$ at level $t$, we can label the vertices in $C(u_{i_1 \ldots i_{t}})$ and that in $C(u_{j_1 \ldots j_{t}})$ independently. The vertices in $C(u_{i_1 \ldots i_{t}})$ are pairwise distance two apart. Apart from these vertices, the only other vertices which are within distance three from the vertices in $N(u_{i_1 \ldots i_{t}})$ are $u_{i_1 \ldots i_{t}}$, $u_{i_1 \ldots i_{t-1}}$ and the vertices in $N(u_{i_1 \ldots i_{t-1}}) \setminus \{u_{i_1 \ldots i_{t}}\}$.

Denote 
\be
\label{eq:defA}
A := f(N(u_{i_1 \ldots i_{t-1}})) \cup \{f(u_{i_1 \ldots i_{t-1}})\},
\ee
\be
\label{eq:defB}
B := [f(u_{i_1 \ldots i_{t}})-(h-1), f(u_{i_1 \ldots i_{t}})+(h-1)]_{\ell}.
\ee
(See Figure \ref{fig:ABC} for an illustration.)
In view of the discussion in the previous paragraph, a label in $\{0, 1, \ldots, \ell - 1\}$ can be used by a vertex of $C(u_{i_1 \ldots i_{t}})$ if and only if it is not in $A \cup B$ and moreover its $\ell$-cyclic distance to every element of $A$ is at least $p$. Since (i) $f(N(u_{i_1 \ldots i_{t-1}}))$ is a $p$-set $\mod~\ell$ with length $m+1$ (by assumption $Q_t$), (ii) $f(u_{i_1 \ldots i_{t}})$ is contained in $f(N(u_{i_1 \ldots i_{t-1}}))$ and is the centre of $B$, and (iii) $h \ge (m+1)p$, we have $f(N(u_{i_1 \ldots i_{t-1}}))  \subseteq [f(u_{i_1 \ldots i_{t}})-mp, f(u_{i_1 \ldots i_{t}})+mp]_{\ell} \subseteq B$ regardless of the location of $f(u_{i_1 \ldots i_{t}})$ in $f(N(u_{i_1 \ldots i_{t-1}}))$. Since $h \ge (m+1)p$, the $\ell$-cyclic distance between each of $f(u_{i_1 \ldots i_{t}})-mp$, $f(u_{i_1 \ldots i_{t}})+mp$ and each of $f(u_{i_1 \ldots i_{t}})-h$, $f(u_{i_1 \ldots i_{t}})+h$ is at least $p$. Therefore, every label in $[f(u_{i_1 \ldots i_{t}})+h, f(u_{i_1 \ldots i_{t}})-h]_{\ell}$ whose $\ell$-cyclic distance to $f(u_{i_1 \ldots i_{t-1}})$ is no less than $p$ can be used by a vertex in $C(u_{i_1 \ldots i_{t}})$. Note that this circular interval contains $\ell - (2h-1) = (m+1)p$ labels including $f(u_{i_1 \ldots i_{t-1}})$. The size $(m+1)p$ is large enough to guarantee the existence of a subset $Y$ of $[f(u_{i_1 \ldots i_{t}})+h, f(u_{i_1 \ldots i_{t}})-h]_{\ell} \setminus \{f(u_{i_1 \ldots i_{t-1}})\}$ with size $m$ such that $Y \cup \{f(u_{i_1 \ldots i_{t-1}})\}$ is a $p$-set $\mod~\ell$, regardless of the location of $f(u_{i_1 \ldots i_{t-1}})$. Assign the $m$ labels in $Y$ to the $m$ vertices of $C(u_{i_1 \ldots i_{t}})$ in a one-to-one manner, so that $f(N(u_{i_1 \ldots i_{t}})) = Y \cup \{f(u_{i_1 \ldots i_{t-1}})\}$ is a $p$-set $\mod~\ell$. Since this can be carried out independently for distinct vertices $u_{i_1 \ldots i_{t}}$ at level $t$, we have proved that we can label all vertices of $\hat{T}_{m, \infty}$ at level $t+1$ such that $P_{t+1}$ and $Q_{t+1}$ hold. 

By induction, we conclude that $\hat{T}_{m, \infty}$ admits a $C(h,p,p)$-labelling of span $\ell$ with properties $P_t$ and $Q_t$ for all $t \ge 1$. Let $I_{u_{i_1 \ldots i_{t-1}}}$ be the shortest circular interval $\mod~\ell$ that contains the $p$-set $f(N(u_{i_1 \ldots i_{t-1}}))$ $\mod~\ell$, where $t \ge 1$. Then $I_{u_{i_1 \ldots i_{t-1}}}$ has the same ends as $f(N(u_{i_1 \ldots i_{t-1}}))$, and as such the proof above shows that for any $u_{i_1 \ldots i_{t}} \in C(u_{i_1 \ldots i_{t-1}})$, $I_{u_{i_1 \ldots i_{t-1}}} \cap I_{u_{i_1 \ldots i_{t}}} = \emptyset$. Thus the $C(h,p,p)$-labelling constructed above is an elegant $C(h,p,p)$-labelling of $\hat{T}_{m, \infty}$ with span $\ell$. Therefore, $\s^*_{h,p,p}(\hat{T}_{m, \infty}) \le \ell$ as claimed in (\ref{hat4}), and this completes the proof of (\ref{eq:c}). 

Since $\tmk$, $T_{m, \infty}$, $\hat{T}_{m, k}$ and $\hat{T}_{m, \infty}$ all have $\De = m+1$, (\ref{eq:d}) follows from (\ref{eq:c}) immediately.  
\end{proof}
\qed
 
An immediate consequence of Theorem \ref{thm:c} is the following result. 

\begin{corollary}
\label{core:c}
Let $T$ be a finite tree with $\diam \ge 3$ and $\De \ge 3$, or an infinite tree with finite maximum degree $\De \ge 3$. Then, for any integer $h \ge \De$,  
\be
\label{eq:c1}
\s_{h,1,1}(T) = \s^*_{h,1,1}(T) = 2h + \De - 1,
\ee
and for finite trees $T$ there exists a linear time exact algorithm for computing $\s_{h,1,1}(T)$. In particular, for $T = \tmk, T_{m, \infty}, \hat{T}_{m, k}$ or $\hat{T}_{m, \infty}$, and any $h \ge m+1$, 
\be
\label{eq:c2}
\s_{h,1,1}(T) = \s^*_{h,1,1}(T) = 2h + m.
\ee
\end{corollary}

\begin{proof}
The truth of (\ref{eq:c1}) and  (\ref{eq:c2}) follows from (\ref{eq:c}) immediately. Following the proof of (\ref{hat4}) and setting $\ell := 2h+\De- 1$, we apply the following algorithm to obtain $f: V(T) \rightarrow \{0, 1, \ldots, \ell - 1\}$:  
\begin{itemize}
\item $f(u_0) = 0$;
\item $f(u_i) = h+i-1$, $1 \le i \le d(u_0)$;
\item assign distinct labels from $\left(2h + [i-1, \De+i-2]_{\ell}\right) \setminus \{0\}$ to the vertices in $f(C(u_i))$, $1 \le i \le d(u_0)$;
\item if for some $t$ with $1 \le t < k$, all vertices of $T$ up to level $t$ have been labelled, then for every vertex $u_{i_1 \ldots i_t}$ of $T$ at level $t$, choose a subset $Y$ of $[f(u_{i_1 \ldots i_t})+h, f(u_{i_1 \ldots i_t})-h]_{\ell} \setminus \{f(u_{i_1 \ldots i_{t-1}})\}$ with size $d(u_{i_1 \ldots i_t})$ such that $Y \cup \{f(u_{i_1 \ldots i_{t-1}})\}$ is a circular interval $\mod~\ell$, and assign distinct labels from $Y$ to the children of $u_{i_1 \ldots i_t}$.
\end{itemize}
It follows from the proof of (\ref{hat4}) that such a set $Y$ always exists and $f$ is an elegant $C(h, 1, 1)$-labelling of $T$ with span $\ell$. Thus, in view of (\ref{eq:c1}), the above is an exact algorithm for producing an optimal $C(h, 1, 1)$-labelling of $T$ in linear time. 
\end{proof}
\qed

\begin{remark}
\label{re:c}
{\em (a) The algorithm in the proof above produces a $C(h, 1, 1)$-labelling $f$ that is both optimal and super elegant in the sense that $f(N(v))$ is a circular interval $\mod~\ell$ for every vertex $v$. In particular, it gives an optimal and super elegant $C(h, 1, 1)$-labelling of $\tmk, T_{m, \infty}, \hat{T}_{m, k}$ or $\hat{T}_{m, \infty}$ when applied to these graphs. 

(b) The algorithm in the proof of Theorem \ref{thm:c} produces an elegant $C(h, p, p)$-labelling $f$ of $\hat{T}_{m, \infty}$ under the condition $h \ge (m+1)p$. Since any finite or infinite tree $T$ with $\diam \ge 3$ and $\De \ge 3$ is a subtree of $\hat{T}_{m, \infty}$ where $m = \De-1$, the restriction of this algorithm to $T$ (that is, the restriction of $f$ to $V(T)$) produces an elegant $C(h, p, p)$-labelling of $T$ whose span is the upper bound in (\ref{eq:c}).

(c) The upper bounds in (\ref{eq:c}) and (\ref{eq:d}) are equivalent as one implies the other.}
\end{remark}

Note that we require $p=1$ in (\ref{eq:c2}). When $p > 1$ and $h$ is sufficiently large, we are able to determine the values of $\s_{h,p,p}$ and $\s^*_{h,p,p}$ for $\tm2$ and $\hat{T}_{m, 2}$, as stated in the next proposition. 

\begin{proposition}
\label{thm:d}
Let $h, p \ge 1$ and $m \ge 2$ be integers. If $h \ge mp$, then 
\be
\label{eq:b}
\s_{h,p,p}(\tm2) = \s^*_{h,p,p}(\tm2) = 2h+mp;
\ee
if $h \ge (m+1)p$, then
\be
\label{eq:b1}
\s_{h,p,p}(\hat{T}_{m, 2}) = \s^*_{h,p,p}(\hat{T}_{m, 2}) = 2h+mp. 
\ee
\end{proposition}

\begin{proof}
Denote $\ell := 2h+mp$. 

(a) We first prove (\ref{eq:b}) under the assumption $h \ge mp$. We first show that $\ell$ is a lower bound for $\s_{h,p,p}(\tm2)$ under this condition. (Note that under the stronger condition $h \ge (m+1)p$ this follows from the lower bound in (\ref{eq:c}).) Suppose otherwise. Then $\tm2$ admits a $C(h,p,p)$-labelling $f: V(\tm2) \rightarrow \{0, 1, \ldots, \ell - 2\}$ with span $\ell - 1$. Without loss of generality we may assume $f(u_0) = 0$. Since $u_1, \ldots, u_m$ are adjacent to $u_0$, we have $f(u_i) \in [h, (\ell-1)-h]_{\ell-1}$ for $1 \le i \le m$. Moreover, since $u_1, \ldots, u_m$ are pairwise distance two apart, their labels should pairwise differ by at least $p$. However, $\{h, h+p, \ldots, h+(m-1)p\}$ is the only subset of $[h, (\ell-1)-h]_{\ell-1}$ whose elements mutually differ by at least $p$. Thus we must have $f(C(u_0)) = \{h, h+p, \ldots, h+(m-1)p\}$. Without loss of generality we may assume $f(u_i) = h+(i-1)p$, $1 \le i \le m$. Since $f(u_1) = h$, we have $f(C(u_1)) \subseteq [2h, \ell - 2]_{\ell-1}$. Moreover, since all vertices in $C(u_1)$ are distance two apart from $u_0$, the $(\ell-1)$-cyclic distances between their labels and $0$ must be at least $p$. Hence $f(C(u_1)) \subseteq [2h, (\ell-1) - p]_{\ell-1}$. Since the $m$ vertices in $C(u_1)$ are distance two apart, they should receive labels that have mutual $(\ell-1)$-cyclic distances at least $p$. However, since $(\ell -1 - p) - 2h = (m-1)p - 1 < mp$, there exists no subset of $[2h, (\ell - 1) - p]_{\ell-1}$ of size $m$ whose elements have mutual $(\ell-1)$-cyclic distances at least $p$. This contradiction shows that $\s_{h,p,p}(\tm2) \ge \ell$. 

Define $f: V(\tm2) \rightarrow \{0, 1, \ldots, \ell - 1\}$ such that
\begin{itemize}
\item $f(u_0) = 0$;
\item $f(u_i) = h+(i-1)p$, $1 \le i \le m$;
\item $f(C(u_i)) = (2h + p[i-1, m+i-1]_{\ell}) \setminus \{0\} = \{2h+(i-1)p, 2h+ip, \ldots, 2h+(m-1)p\} \cup \{p, 2p, \ldots, (i-1)p\}$, $1 \le i \le m$.
\end{itemize}
Using the condition $h \ge mp$, one can verify that the $\ell$-cyclic distance between $h+(i-1)p$ and every label in $f(C(u_i))$ is at least $h$, and the $\ell$-cyclic distance between every $h+(j-1)p$ with $1 \le j \le m$ and $j \ne i$ and every label in $f(C(u_i))$ is at least $p$. Thus $f$ is a $C(h,p,p)$-labelling of $\tm2$ with span $\ell$. Moreover, $f(N(u_0)) \subseteq I_0 := [h, h+(m-1)p]_{\ell}$, $f(N(u_i)) \subseteq I_i := [2h+p(i-1), 2h+p(m+i-1)]_{\ell}$, and $f(N(u_{ij})) = I_{ij} := \{h+(i-1)p\}$. It is evident that $I_0 \cap I_i = \emptyset$ and $I_i \cap I_{ij} = \emptyset$ for $1 \le i, j \le m$. Therefore, $f$ is an elegant $C(h,p,p)$-labelling of $\tm2$ and so $\s^*_{h,p,p}(\tm2) \le \ell$. Thus, by (\ref{eq:elegant}), $\ell \le \s_{h,p,p}(\tm2) \le \s^*_{h,p,p}(\tm2) \le \ell$, and (\ref{eq:b}) follows. 


(b) Suppose $h \ge (m+1)p$. Then we have $\s_{h,p,p}(\hat{T}_{m, 2}) \ge \ell$ by (\ref{eq:d}). Define $f: V(\hat{T}_{m, 2}) \rightarrow \{0, 1, \ldots, \ell - 1\}$ in exactly the same way as in (a) above except that $i$ now ranges from $1$ to $m+1$. Using the condition $h \ge (m+1)p$, one can verify that $f$ is an elegant $C(h,p,p)$-labelling of $\hat{T}_{m, 2}$ with span $\ell$. Therefore, $\s^*_{h,p,p}(\hat{T}_{m, 2}) \le \ell$ and (\ref{eq:b1}) follows. 
\end{proof}
\qed


\section{$C(h,1,1)$-labellings of trees}
\label{sec:cyc}

In this section we will first give the precise values of $\s_{h,1,1}(T)$ and $\s^*_{h,1,1}(T)$ for $T = \tmk$ ($k \ge 4$), $T_{m,\infty}$, $\hat{T}_{m, k}$, $\hat{T}_{m, \infty}$. Based on this we will then give a sharp upper bound on $\s^*(T)$ for general $T$ and any $h \ge 1$, and a $(4\De + 1)/(2\De + 4)$-approximation algorithm for the $C(h, 1, 1)$-labelling problem for finite trees. Alternatively, similar to what we did in the previous two sections, we could state our results for general trees first and then obtain the values of $\s_{h,1,1}$ and $\s^*_{h,1,1}$ for the four trees above as special cases. Nevertheless, the two treatments are equivalent, and for both the key is an algorithm that produces optimal `super elegant' $C(h, 1, 1)$-labellings of $T_{m, \infty}$ and $\hat{T}_{m, \infty}$. An upper bound on $\s^*(T)$ for general $T$ then follows by viewing $T$ as a subgraph of an appropriate $\hat{T}_{m, \infty}$, akin to the treatment in the proof of Theorem \ref{thm:c}.

\begin{proposition}
\label{thm:a}
Let $m \ge 2$ and $h \ge 1$ be integers. Then
\be 
\label{ck2}
\s_{h,1,1}(\tm2) = \s^*_{h,1,1}(\tm2) = \max\{h+2m, 2h+m\}
\ee
\be 
\label{hat2}
\s_{h,1,1}(\hat{T}_{m, 2}) = \s^*_{h,1,1}(\hat{T}_{m, 2}) = \max\{h+2m+1, 2h+m\}.
\ee
\end{proposition}

\begin{proof} 
(a) If $h \ge m$, then (\ref{ck2}) follows from (\ref{eq:b}). 

Assume $h \le m$. Let $\ell := h+2m$. By (\ref{eq:lvsc1}) and (\ref{eq:thmhpq}), $\s_{h,1,1}(\tm2) \ge \l_{h, 1, 1}(\tm2) + 1 = \ell$. Define $f: V(\tm2) \rightarrow \{0, 1, \ldots, \ell-1\}$ such that
\begin{itemize} 
\item $f(u_0) = 0$;
\item $f(u_i) = h+i-1$, $1 \le i \le m$;
\item $f(C(u_i)) = [h+m, h+2m-1]_{\ell}$, $1 \le i \le m-h+1$; 
\item $f(C(u_i)) = [2h+i-1, 2h+m+i-1]_{\ell} \setminus \{0\}$, $m-h+2 \le i \le m$.
\end{itemize}
Since $h \le m$, one can verify that $f$ is an elegant $C(h, 1, 1)$-labelling of $\tm2$ with span $\ell$. Thus, $\s^*_{h,1,1}(\tm2) \le \ell$, which together with $\ell \le \s_{h,1,1}(\tm2) \le \s^*_{h,1,1}(\tm2)$ implies (\ref{ck2}).

(b) If $h \ge m+1$, then (\ref{hat2}) follows from (\ref{eq:b1}). 

Assume $h \le m$. Set $\ell := h+2m+1$. Then $\s_{h,1,1}(\hat{T}_{m, 2}) \ge \l_{h,1,1}(\hat{T}_{m, 2}) + 1 = \ell$ by (\ref{eq:lvsc1}) and (\ref{eq:hat}). Define $f: V(\hat{T}_{m, 2}) \rightarrow \{0, 1, \ldots, \ell-1\}$ such that
\begin{itemize} 
\item $f(u_0) = 0$;
\item $f(u_i) = h+i-1$ for $1 \le i \le m+1$;
\item $f(C(u_i)) = [h+m+1, h+2m]_{\ell}$, $1 \le i \le m-h+2$;
\item $f(C(u_i)) = [2h+i-1, 2h+m+i-1]_{\ell} \setminus \{0\}$, $m-h+3 \le i \le m+1$.
\end{itemize}
Using the condition $h \le m$, one can show that $f$ is an elegant $C(h, 1, 1)$-labelling of $\hat{T}_{m, 2}$ with span $\ell$. Thus $\s^*_{h,1,1}(\hat{T}_{m, 2}) \le \ell$ and (\ref{hat2}) follows.
\end{proof}
\qed
 
\begin{theorem}
\label{thm:ck3}
Let $h \ge 1$ and $m, k \ge 2$ be integers. Then, for $T = \tmk$ ($k \ge 4$), $T_{m, \infty}$, $\hat{T}_{m, k}$ or $\hat{T}_{m, \infty}$, 
\be
\label{eq:bdd}
\s_{h,1,1}(T) = \s^*_{h,1,1}(T) = \max\{h+2m+1, 2h+m\}.
\ee
\end{theorem}

\begin{proof}
Observe that, if $h \ge m+1$, then (\ref{eq:bdd}) is reduced to (\ref{eq:c2}). In what follows we assume $h \le m$ and set $\ell := h+2m+1$. 

(a) We prove (\ref{eq:bdd}) for $T = \tmk$ ($k \ge 4$) or $T_{m,\infty}$ first. We have $\s_{h,1,1}(T_{m,\infty}) \ge \s_{h,1,1}(\tmk) \ge \l_{h, 1, 1}(\tmk)+1 = \ell$ by (\ref{eq:lvsc1}) and (\ref{eq:hat}). 

To prove $\s^*_{h,1,1}(T_{m,\infty}) \le \ell$, we recursively label the vertices of $T_{m,\infty}$ from lower to higher levels such that for each vertex $u$ the labels assigned to the neighbours of $u$ (including its parent if $u \ne u_0$) form a circular interval $\mod~\ell$. More explicitly, we construct an elegant $C(h,1,1)$-labelling $f$ of $T_{m,\infty}$ with span $\ell$, beginning with the following labels for vertices up to level $2$:
\begin{itemize} 
\item $f(u_0) = 0$;
\item $f(u_i) = h+i-1$, $1 \le i \le m$;
\item $f(C(u_i)) = [h+m+1, h+2m]_{\ell}$, $1 \le i \le m-h+1$;
\item $f(C(u_i)) = [2h+i-1, 2h+m+i-1]_{\ell} \setminus \{0\}$, $m-h+2 \le i \le m$.
\end{itemize}
It can be verified that the $C(h,1,1)$-conditions $\mod~\ell$ are satisfied among vertices of $T_{m,\infty}$ up to level 2. 

Note that $f(N(u_0))$ is a circular interval $\mod~\ell$, and each $f(N(u_i))$ is a circular interval $\mod~\ell$ of length $m+1$. Assume inductively that, for some $t \ge 2$, all vertices of $T_{m,\infty}$ up to level $t$ have been labelled such that:
\begin{itemize}
\item[] $P_t$: the $C(h,1,1)$-conditions $\mod~\ell$ are satisfied among vertices up to level $t$;  
\item[] $Q_t$: for every vertex $u_{i_1 \ldots i_{t-1}}$ at level $t-1$, $f(N(u_{i_1 \ldots i_{t-1}}))$ is a circular interval $\mod~\ell$. 
\end{itemize}
We proceed to prove that we can label all vertices of $T_{m,\infty}$ at level $t+1$ such that $P_{t+1}$ and $Q_{t+1}$ hold. 

A label can be used by a vertex of $C(u_{i_1 \ldots i_{t}})$ if and only if it is in the set 
$$
X := \{0, 1, \ldots, \ell - 1\} \setminus (A \cup B),
$$ 
where $A$ and $B$ are as defined in (\ref{eq:defA}) and (\ref{eq:defB}), respectively, with respect to the present integer $\ell$ (that is, $\ell = h+2m+1$). See Figure \ref{fig:ABC} for an illustration, with the understanding that this time $f(N(u_{i_1 \ldots i_{t-1}}))$ may not be a subset of $B$. 

By our assumption $Q_t$, $f(N(u_{i_1 \ldots i_{t-1}}))$ is a circular interval $\mod~\ell$ of length $m+1$. In addition, it has at least one common element with $B$, namely $f(u_{i_1 \ldots i_{t}})$. Since $f(u_{i_1 \ldots i_{t-1}}) \not \in B$, it follows that $A \cup B$ consists of a circular interval $\mod~\ell$ and a single integer $f(u_{i_1 \ldots i_{t-1}})$. Therefore, $X$ is either the union of two circular intervals $\mod~\ell$ separated by $f(u_{i_1 \ldots i_{t-1}})$ or a single circular interval $\mod~\ell$, the latter occuring if and only if $A \cup B$ itself is a circular interval $\mod~\ell$. 

Since $f(u_{i_1 \ldots i_{t}}) \in f(N(u_{i_1 \ldots i_{t-1}}))$ is in the centre of $B$, we have $|A \cap B| \ge h$, and equality holds if and only if either $f(u_{i_1 \ldots i_{t}})$ is one end of the circular interval $f(N(u_{i_1 \ldots i_{t-1}}))$ $\mod~\ell$, or $h = 1$ (and so $B = \{f(u_{i_1 \ldots i_{t}})\} \subset f(N(u_{i_1 \ldots i_{t-1}}))$), or $f(N(u_{i_1 \ldots i_{t-1}})) \subseteq B$ and $h = m + 1$. Thus $|A \cup B| = |A| + |B| - |A \cap B| = (m+2) + (2h-1) - |A \cap B| \le h+m+1$, and therefore 
$$
|X| = \ell - |A \cup B| = (h+2m+1) - |A \cup B| \ge m.
$$ 
Since $X$ is the union of at most two circular intervals $\mod~\ell$ possibly separated by $f(u_{i_1 \ldots i_{t-1}})$, it follows that there exists a subset $Y$ of $X$ with $|Y| = m$ such that $Y \cup \{f(u_{i_1 \ldots i_{t-1}})\}$ is a circular interval $\mod~\ell$. Label the vertices of $C(u_{i_1 \ldots i_{t}})$ by the elements of $Y$ in a bijective manner, so that $f(N(u_{i_1 \ldots i_{t}})) = Y \cup \{f(u_{i_1 \ldots i_{t-1}})\}$ is a circular interval $\mod~\ell$. These labels for $C(u_{i_1 \ldots i_{t}})$ do not violate the $C(h,1,1)$-conditions $\mod~\ell$ with vertices of level at most $t$. Since this holds for every $u_{i_1 \ldots i_{t}}$ at level $t$, and since we can label the sets $C(u_{i_1 \ldots i_{t}})$ independently, we can label all vertices of $T_{m,\infty}$ up to level $t+1$ such that $P_{t+1}$ and $Q_{t+1}$ are satisfied. 

By induction, we have proved that $T_{m,\infty}$ admits a $C(h,1,1)$-labelling of span $\ell$ with properties $P_t$ and $Q_t$ for all $t \ge 1$. One can see from the construction above that this is an elegant $C(h,1,1)$-labelling. Therefore, $\s^*_{h,1,1}(T_{m,\infty}) \le \ell$, which together with $\ell \le \s_{h,1,1}(T_{m,\infty}) \le \s^*_{h,1,1}(T_{m,\infty})$ yields (\ref{eq:bdd}) for $T = T_{m,\infty}$.

The restriction of the labelling above to the vertices of $T_{m,\infty}$ up to level $k$ is an elegant $C(h,1,1)$-labelling of $\tmk$, implying $\s^*_{h,1,1}(\tmk) \le \ell$ for any $k \ge 2$. This together with $\ell \le \s_{h,1,1}(\tmk) \le \s^*_{h,1,1}(\tmk)$ gives (\ref{eq:bdd}) for $T = \tmk$ ($k \ge 4$).  
 
(b) Now we prove (\ref{eq:bdd}) for $T = \hat{T}_{m, k}$ or $\hat{T}_{m, \infty}$. We have $\s_{h,1,1}(\hat{T}_{m, \infty}) \ge \s_{h,1,1}(\hat{T}_{m, k}) \ge \l_{h,1,1}(\hat{T}_{m, k}) + 1 = \ell$ by (\ref{eq:lvsc1}) and (\ref{eq:hat}). Label the vertices of $\hat{T}_{m, \infty}$ up to level $2$ in exactly the same way as in part (b) of the proof of Proposition \ref{thm:a}. Note that for this labelling $f$, $f(N(u_0))$ is a circular interval $\mod~\ell$, and $f(N(u_i))$ is a circular interval $\mod~\ell$ for $1 \le i \le m+1$. Based on this and by induction, we can prove $\s^*_{h,1,1}(\hat{T}_{m,\infty}) \le \ell$ in exactly the same way as in (a) above. Similar to (a), from this we obtain (\ref{eq:bdd}) for $T =  \hat{T}_{m,\infty}$ or $\hat{T}_{m,k}$.
\end{proof}
\qed

\begin{remark}
\label{re:c1}
{\em
In view of Remark \ref{re:c}(a), the proofs of Theorems \ref{thm:c} and \ref{thm:ck3} together give an algorithm for constructing an optimal and super elegant $C(h,1,1)$-labelling of $T_{m,\infty}$, whose restriction to the vertices up to level $k$ is an optimal and super elegant $C(h,1,1)$-labelling of $\tmk$ ($k \ge 4$). The same can be said for $\hat{T}_{m, \infty}$ and $\hat{T}_{m, k}$.
}
\end{remark}

\begin{theorem}
\label{thm:s11forD}
Let $T$ be a finite tree with $\diam \ge 3$ and $\De \ge 3$, or an infinite tree with finite maximum degree $\De \ge 3$. Then, for any integer $h \ge 1$, 
\begin{equation}
\label{eq:s11forD}
\max\{\De_2,  2h + \De - 1\} \le \s_{h,1,1}(T) \le \s^*_{h,1,1}(T) \le \max\{h+2\De-1, 2h+\De-1\}.
\end{equation}
In particular, if $h \le \De$ and $T$ contains $T_{\De-1, 2}$ as a subtree, then 
\be
\label{eq:s11forD1}
h+2\De-2 \le \s_{h,1,1}(T) \le \s^*_{h,1,1}(T) \le h+2\De-1;
\ee
if $h \le \De$ and $T$ contains $\hat{T}_{\De-1, 2}$ as a subtree, then 
\be
\label{eq:s11forD2}
\s_{h,1,1}(T) = \s^*_{h,1,1}(T) = h+2\De-1.
\ee
\end{theorem}

\begin{proof} 
By (\ref{eq:lvsc1}) and (\ref{eq:h11}), we have $\s_{h,1,1}(T)  \ge \l_{h,1,1}(T)+1 \ge \De_2$. The other lower bound $2h + \De - 1$ is a special case of the lower bound in (\ref{eq:c}) (which itself does not need the condition $h \ge \De p$). 

Taking $T$ as rooted at a neighbour of a maximum degree vertex, we can extend $T$ to $\hat{T}_{\De-1, k}$, where $k \ge 2$ is the maximum level of a vertex of $T$. Thus, by (\ref{hat2}) and (\ref{eq:bdd}), we have $\s_{h,1,1}(T) \le \s^*_{h,1,1}(T) \le \s^*_{h,1,1}(\hat{T}_{\De-1, k}) = \max\{h+2\De-1, 2h+\De-1\}$ as stated in (\ref{eq:s11forD}). 

If $h \le \De$ and $T$ contains $T_{\De-1, 2}$ as a subtree, then $\s_{h,1,1}(T) \ge \s_{h,1,1}(T_{\De-1, 2}) = \max\{h+2\De-2, 2h+\De-1\} = h+2\De-2$ by (\ref{ck2}), which together with the upper bound in (\ref{eq:s11forD}) implies (\ref{eq:s11forD1}).

If $h \le \De$ and $T$ contains $\hat{T}_{\De-1, 2}$ as a subtree, then $h+2\De-1 \le \s_{h,1,1}(T)$ by (\ref{hat2}) and (\ref{eq:bdd}), and $\s^*_{h,1,1}(T) \le h+2\De-1$ by (\ref{eq:s11forD}), leading to (\ref{eq:s11forD2}).
\end{proof}
\qed

Note that, unsurprisingly, (\ref{eq:s11forD}) yields the same result as (\ref{eq:c1}) when $h \ge \De$. Since in this case a linear time exact algorithm for the $C(h,1,1)$-labelling problem was given in Corollary \ref{core:c}, we may assume $h < \De$ in the following result. 

\begin{theorem}
\label{thm:1.4}
Given any integer $h \ge 1$, there exists a linear-time $(1+\ve)$-factor approximation algorithm for the $C(h,1,1)$-labelling problem for the class of finite trees, where $\ve = (2\De - 3)/(2\De + 4)$. Moreover, this algorithm is a $1.4$-factor approximation algorithm for finite trees with $\De_2 = 2\De - 1$ or $2\De$. 
\end{theorem}

\begin{proof}
It suffices to consider trees $T$ with $\diam \ge 3$ and $\De \ge 3$. We may assume $h < \De$ as mentioned above. 

We view $T$ as a subtree of $T_{\De-1, k}$ for some $k \ge 2$, with root $u_0$ at a neighbour of a maximum degree vertex (if $\De_2 \le 2\De - 1$) or at a neighbour of one end of a heavy edge other than the other end (if $\De_2 = 2\De$). This enables us to follow the proof of Theorem \ref{thm:ck3} to give a $C(h,1,1)$-labelling with span the upper bound in (\ref{eq:s11forD}). (As seen in the second last paragraph in the proof of Theorem \ref{thm:ck3}, this upper bound holds for any $k \ge 2$.) We assume $k \ge 3$ in the following since the simpler case $k=2$ can be dealt with similarly by following the proofs of (\ref{eq:b}) and (\ref{ck2}).

Denote $\ell := h + 2\De - 1$. Following the proof of Theorem \ref{thm:ck3}, we apply the following algorithm to obtain $f: V(T) \rightarrow \{0, 1, \ldots, \ell - 1\}$:
\begin{itemize}
\item $f(u_0) = 0$;
\item $f(u_i) = h+i-1$, $1 \le i \le d(u_0)$;
\item for $1 \le i \le \max\{0, d(u_0) - h + 1\}$, assign distinct labels from $[h+\De, h+2\De-2]_{\ell}$ to the vertices in $C(u_i)$; and for $\max\{1, d(u_0) - h + 2\} \le i \le \ldots, d(u_0)$, assign distinct labels from $[2h+i-1, 2h+\De+i-2]_{\ell} \setminus \{0\}$ to the vertices in $C(u_i)$;
\item assume, for some $t$ with $1 \le t < k$, all vertices of $T$ up to level $t$ have been labelled, then for every vertex $u_{i_1 \ldots i_t}$ of $T$ at level $t$ choose a subset $Y$ of $X$ with $|Y| = d(u_{i_1 \ldots i_t})$ such that $Y \cup \{f(u_{i_1 \ldots i_{t-1}})\}$ is a circular interval $\mod~\ell$, and then assign distinct labels from $Y$ to the children of $u_{i_1 \ldots i_{t}}$, where $X := \{0, 1, \ldots, \ell - 1\} \setminus (A \cup B)$ with $A := f(N(u_{i_1 \ldots i_{t-1}})) \cup \{f(u_{i_1 \ldots i_{t-1}})\}$ and $B := [f(u_{i_1 \ldots i_{t}})-(h-1), f(u_{i_1 \ldots i_{t}})+(h-1)]_{\ell}$.
\end{itemize}
It follows from the proof of Theorem \ref{thm:ck3} that such a set $Y$ always exists and $f$ is a $C(h, 1, 1)$-labelling of $T$ with span $\ell$. Obviously, the algorithm above runs in linear time. In view of (\ref{eq:s11forD}), its performance ratio is at most
$$
\frac{h + 2\De - 1}{\max\{\De_2,  2h + \De - 1\}}.
$$
It can be verified that this ratio is at most $1+\ve$, no matter whether $2h \le \De_2 - \De + 1$ or not. In particular, if $\De_2 = 2\De - 1$ or $2\De$, then one can show that this ratio is at most $7/5$.  
\end{proof}
\qed


\section{Concluding remarks}
\label{sec:concl}

In Theorem \ref{thm:h11} we give a linear time $(\De_2 + \De - 2)/(\De_2 - 1)$-approximation algorithm for the $L(h, p, p)$-labelling problem for finite trees. A polynomial time approximation algorithm for this problem with better performance ratio is desirable.  
Similarly, it would be good if one can devise a polynomial time approximation algorithm for the $C(h, 1, 1)$-labelling problem for finite trees whose performance ratio is smaller than the one in Theorem \ref{thm:1.4}. 

We do not know the exact values of $\s_{h, p, p}(T)$ and $\s^*_{h, p, p}(T)$ for $T = \tmk$,$T_{m,\infty}$, $\hat{T}_{m, k}$ or $\hat{T}_{m,\infty}$ even when $h \ge (m+1)p > m+1$. Determining these values or improving the upper bound in (\ref{eq:d}) will help understand $\s_{h, p, p}$ and $\s^*_{h, p, p}$ for general trees. 

Define $\bar{\l}_{h, 1, 1}(G)$ and $\bar{\s}_{h, 1, 1}(G)$ to be the minimum integer $\ell$ such that $G$ admits a super elegant $L(h, 1, 1)$-labelling, and a super elegant $C(h, 1, 1)$-labelling, with span $\ell$, respectively, or $\infty$ if $G$ does not admit such a labelling. Clearly, $\l^*_{h, 1, 1}(G) \le \bar{\l}_{h, 1, 1}(G)$ and $\s^*_{h, 1, 1}(G) \le \bar{\s}_{h, 1, 1}(G)$. As mentioned in Remark \ref{re:super}, for $T = \tmk$ ($k \ge 4$), $T_{m, \infty}$, $\hat{T}_{m, k}$ or $\hat{T}_{m, \infty}$, $T$ admits a super elegant $L(h,1,1)$-labelling and moreover $\bar{\l}_{h, 1, 1}(T) = \l^*_{h, 1, 1}(T) = \l_{h, 1, 1}(T) = h + 2m$. In view of Remark \ref{re:c}, for any finite or infinite tree $T$ with $\diam \ge 3$ and $3 \le \De \le h$, $T$ admits a super elegant $C(h,1,1)$-labelling and moreover $\bar{\s}_{h, 1, 1}(T) = \s^*_{h, 1, 1}(T) = \s_{h, 1, 1}(T) = 2h + \De - 1$; in particular, this holds for $\tmk$, $T_{m, \infty}$, $\hat{T}_{m, k}$ and $\hat{T}_{m, \infty}$. Since any tree $T$ can be viewed as a subtree of some $\hat{T}_{m, \infty}$, and since the restriction of an elegant $L(h, 1, 1)$ or $C(h, 1, 1)$-labelling of $\hat{T}_{m, \infty}$ to $T$ is an elegant $L(h, 1, 1)$ or $C(h, 1, 1)$-labelling of $T$, from $\l^*_{h, 1, 1}(\hat{T}_{m, \infty}) < \infty$ and $\s^*_{h, 1, 1}(\hat{T}_{m, \infty}) < \infty$ it follows that any tree admits an elegant $L(h, 1, 1)$-labelling as well as an elegant $C(h, 1, 1)$-labelling. (This is also implied in \cite[Theorem 4]{FGKLP} which asserts that $\l^*_{h, 1, 1}$ and $\s^*_{h, 1, 1}$ can be computed in polynomial time for any tree and any integer $h \ge 1$.) However, from $\bar{\l}_{h, 1, 1}(\hat{T}_{m, \infty}) < \infty$ and $\bar{\s}_{h, 1, 1}(\hat{T}_{m, \infty}) < \infty$ we cannot derive that any tree admits a super elegant $L(h, 1, 1)$-labelling and a super elegant $C(h, 1, 1)$-labelling, because it is not clear whether the restriction of a super elegant labelling to a subtree is also super elegant. Therefore, we may ask under what conditions a tree admits a super elegant $L(h, 1, 1)$-labelling or a super elegant $C(h, 1, 1)$-labelling. Also it would be interesting to characterise those trees $T$ such that $\bar{\l}_{h, 1, 1}(T) = \l^*_{h, 1, 1}(T)$, and those trees such that $\bar{\s}_{h, 1, 1}(T) = \s^*_{h, 1, 1}(T)$. 
  
\medskip

\noindent {\bf Acknowledgments}~~We appreciate the anonymous referees for their helpful comments. Zhou was supported by a Future Fellowship (FT110100629) of the Australian Research Council.

\end{document}